\documentclass[a4paper,11pt]{amsart}
\usepackage[utf8]{inputenc}
\usepackage{indentfirst, amsfonts, amsmath, amsthm, amssymb, amscd}
\usepackage{amsmath,amsfonts,amscd,bezier}
\usepackage{hyperref}
\usepackage{color}
\usepackage[active]{srcltx}
\usepackage{mathrsfs}

\newtheorem{maintheorem}{Theorem}

\newtheorem{theorem}{Theorem}[section]

\newtheorem{corollary}[theorem]{Corollary}
\newtheorem{proposition}[theorem]{Proposition}
\newtheorem{lemma}[theorem]{Lemma}
\newtheorem{definition}[theorem]{Definition}

\theoremstyle{remark}

\title{On the Bernoulli property for certain partially hyperbolic diffeomorphisms}
\author{G. Ponce}
\address{Departamento de Matem\'atica,
  IMECC-UNICAMP Campinas-SP, Brazil.}
  \email{gaponce@ime.unicamp.br}
\author{A. Tahzibi} 
\address{Departamento de Matem\'atica,
  ICMC-USP S\~{a}o Carlos-SP, Brazil.}
\email{tahzibi@icmc.usp.br}
\author{R. Var\~{a}o} 
\address{Departamento de Matem\'atica,
  IMECC-UNICAMP Campinas-SP, Brazil.}
\email{regisvarao@ime.unicamp.br}

\date{}                                         
\begin{document}
\maketitle

\begin{abstract}
We address the classical problem of equivalence between Kolmogorov and Bernoulli property of smooth dynamical systems. In a natural class of  volume preserving partially hyperbolic diffeomorphisms homotopic to Anosov (``derived from Anosov")  on 3-torus, we prove that Kolmogorov and Bernoulli properties are equivalent. 

In our approach, we propose to study the conditional measures of volume along central foliation to recover fine ergodic properties for partially hyperbolic diffeomorphisms. As an important consequence we obtain that there exists an almost everywhere conjugacy between any volume preserving derived from Anosov diffeomorphism of 3-torus and its linearization. 

Our results also hold in higher dimensional case when central bundle  is one dimensional and stable and unstable foliations are quasi-isometric.

\end{abstract}

\section{Introduction}

One of the main goals of smooth ergodic theory is to describe or understand the behavior of typical orbits $\{ f^n(x)\}_{n \in \mathbb Z}$ of a given measure preserving diffeomorphism $f: M \rightarrow M$ on a Riemannian manifold $M$. Several topological and metric invariants where introduced in the theory in order to classify the dynamics of flows and diffeomorphisms by detecting, in some sense, the amount of chaoticity of a dynamical system. Two very effective and well-known invariants are the topological and metric entropy. 
If the metric entropy of a $\mu$-measure preserving automorphism $f$ is greater then the metric entropy of $g$ with respect to a $g$-invariant measure $\nu$, we can say that from the ergodic point of view the orbits of $f$ have a richer behavior compared to the orbits of $g$. 

In the seventies, D. Ornstein proved that the metric entropy is a complete invariant for the class of Bernoulli shifts, that is, two Bernoulli shifts with the same metric entropy are isomorphic \cite{Ornstein70,O70}. Bernoulli shifts have a very strong chaotic property: they have completely positive entropy. That means that the metric entropy of any non trivial partition is positive.

In fact systems with completely positive entropy have a special name, they are called Kolmogorov. A natural question is whether every Kolmogorov system is isomorphic to a Bernoulli system. In 1973 D. Ornstein \cite{Ornstein} gave an example of a systems which is Kolmogorov but not Bernoulli (i.e. not isomorphic to a Bernoulli shift).  Later, in 1982 S. Kalikow \cite{Kalikow} exhibited another (much more natural) example of a Kolmogorov but not Bernoulli system (see for a survey on equivalence in ergodic theory by J. P. Thouvenot \cite{JPT}.)

However, the first smooth example of a Kolmogorov but not Bernoulli system was given by A. Katok \cite{Katok} in 1980.  However, It is still not known if it possible to give a smooth example on a manifold with dimension three which is Kolmogorov but not Bernoulli. 
We mention also that D. Rudolph \cite{Rud} had constructed a smooth example of Kolmogorov but not (loosely) Bernoulli similar to Kalikow deep $(T, T^{-1})$ example. A  recent work of T. Austin \cite{Aus} goes further and find a continuum of distinct non-Bernoulli and Kolmogorov automorphisms.

A huge variety of natural transformations worked in the theory were proved to be Kolmogorov and most of them where proved out to be Bernoulli. Y. Katznelson \cite{K}, using harmonic analysis methods proved that ergodic automorphisms of $\mathbb T^n$ are Bernoulli. Using a geometrical approach D. Ornstein and B. Weiss \cite{OW} proved that geodesic flows on negatively curved manifolds are Bernoulli. Later, several authors used the technique of Ornstein-Weiss to obtain the Bernoulli properties in contexts where there is a presence of some hyperbolic structure, for example, M. Ratner \cite{Ratner1} proved the Bernoulli property for Anosov flows preserving a $u$-Gibbs measure and Y. Pesin \cite{YP3} extended Ornstein-Weiss's argument to prove that for $C^{1+\alpha}$ nonuniformly hyperbolic diffeomorphisms, the Kolmogorov and the Bernoulli property are equivalent. In 1996 N. Chernov and C. Haskell \cite{CH} extended the equivalence of the Kolmogorov and Bernoulli property for certain nonuniformly hyperbolic maps and flows (which can possibly have singularities). See also a recent result of Ledrappier, Lima and Sarig \cite{LLS} where they study Bernoulli property for measure of maximal entropy of flows in three dimensional manifolds. 

As we do not intend to do an extensive survey on the equivalence of the Kolmogorov and Bernoulli property along the theory, we just remark that the great majority of the Bernoulli property was obtained as a consequence of the works cited above. 

Turning our attention to the smooth ergodic theory context, a large class of transformations which are important in the theory do not have a complete hyperbolic structure. For instance, the time-one map of a geodesic flow in a negatively curved surface does not have a complete hyperbolic structure since it has an invariant direction where it is an isometry. This map is a particular case of a partially hyperbolic diffeomorphism. A big problem in smooth ergodic theory is to determine under which conditions a partially hyperbolic diffeomorphism is $C^r$-stably ergodic , $r\geq 2$ (see Pugh-Shub conjecture \cite{BW}). The most general result in this direction so far was obtained by A. Wilkinson and K. Burns \cite{BW}, where they proved that a volume preserving $C^2$ partially hyperbolic diffeomorphism satisfying two conditions, namely center-bunching and essential accessibility, is Kolmogorov (see also\cite{HHU1} and survey \cite{HHUsurvey} ). This result gives a flavor of the importance to understand the relation between the Kolmogorov and the Bernoulli property in the presence of partially hyperbolic structures (see the question posed by K. Burns on \cite{problemlist}). 
\subsection{Statement of results}

The present work contributes toward the equivalence problem of Kolmogorov and Bernoulli systems in partially hyperbolic context (Theorem \ref{theo:main}).  We consider partially hyperbolic diffeomorphism in the homotopy class of Anosov automorphisms of $\mathbb{T}^3$ (under some natural hypothesis we also treat higher dimensional tori case. 
See Theorems \ref{theo:main.C}, \ref{theo:main.D}). We call this class of diffeomorphisms as derived from Anosov systems (see Definitions \ref{defi:DA}, \ref{defi:linearizacao}).

Our first result is:

\begin{maintheorem}\label{theo:main}
 Let $f:\mathbb T^3 \rightarrow \mathbb T^3$ be a $C^2$ volume preserving derived from Anosov diffeomorphism with linearization $A:\mathbb T^3 \rightarrow \mathbb T^3$. Assume that $f$ is Kolmogorov and one of the following occurs:
 \begin{enumerate}
 \item $\lambda^c_A <0$ and $\mathcal F^{cs}$ is absolutely continuous, or
 \item $\lambda^c_A>0$ and $\mathcal F^{cu}$ is absolutely continuous.
 \end{enumerate}
Then $f$ is Bernoulli.
\end{maintheorem}

Let us emphasize that  absolute continuity condition  turns out to be important to prove Bernoulli property of smooth measures. \footnote{We would like to observe that the Remark 2.4 in \cite{Brin}  is not correct as it is written. More precisely, the arguments on that paper are not enough to prove Bernoulli property without using absolute continuity of center-stable or center-unstable foliations.}

Although it seems a difficult problem, we conjecture that there may exist Kolmogorov but not Bernoulli partially hyperbolic diffeomorphisms in three dimensional manifolds. 

The novelty of the proof of our result is that it introduces a new approach using disintegration of measures to understand fine ergodic properties of smooth dynamics.

Let $f$ be a  partially hyperbolic diffeomorphism with center foliation.  

\begin{definition} \label{virtual}
An $f$-invariant measure $\mu$ is called virtually hyperbolic if there exists a full measurable invariant subset $Z$ such that $Z$ intersects each center leaf in at most one point.
\end{definition}

The above definition was given in \cite{LS} in the context of algebraic automorphisms and the existence of such measures in partially hyperbolic diffeomorphism also had been noticed before (see for instance\cite{SW00}, \cite{PTV}).  
If $\mu$ is virtually hyperbolic, then the central foliation is measurable with respect to $\mu$ and conditional measures along center leaves are (mono-atomic) Dirac measures.  Indeed the partition into central leaves is equivalent to the partition into points.

Given a derived from Anosov diffeomorphism, it is well known that $f$ is semi-conjugated to a linear Anosov diffeomorphism $A$ on torus by a semi-conjugacy $h$. In the following Theorem, the set of points where $h$ fails to be injective is denoted by $\mathcal C$ (for details see \S \ref{sec:semi-conjugacy}).

\begin{maintheorem}\label{theo:main.B}
Let $f:\mathbb T^3 \rightarrow \mathbb T^3$ be a $C^2$ volume $(m)$ preserving derived from Anosov diffeomorphism with $h$ a semi conjugacy to linear Anosov diffeomorphism. Then, $h$ is $m-$almost everywhere injective. More precisely, the following dichotomy is valid:
\begin{itemize}
\item Either the set $\mathcal C$ has zero volume, or
\item $\mathcal C$ has full measure and $(f, m)$ is virtually hyperbolic. 
\end{itemize}
In the latter case,  $(f, m)$ is Kolmogorov.
\end{maintheorem}

So, observe that $h$ is  injective in $\mathbb{T}^3 \setminus \mathcal{C}.$ If it happens that $m(\mathcal{C}) =1$, then by second part of the dichotomy of above theorem there exists a full measurable subset (of atoms) which intersects each leaf in exactly one point and $h$ restricted to this subset is injective. This motivates the following definition of ``essential injectivity domain" of $h.$
\begin{definition} \label{defi:X}
We define the essential injectivity domain $X$ of $h$ as follows: If $m(\mathcal C) = 0$ define
 \[X := \mathbb T^3 \setminus \mathcal C.\]
Otherwise, define
\[X = \text{ set of atoms.} \]
\end{definition}

We remark that it is not known whether all volume preserving derived from Anosov diffeomorphisms are ergodic. See results of Hammerlindl-Ures in this direction \cite{HU}. By the way, in the above theorem (\ref{theo:main.B}) we are not assuming any ergodicity assumption.

Observe that our above theorem has independent interest, as it shows isomorphisms between volume measure and its image under semi conjugacy. We recall a result of Buzzi-Fisher \cite{BF} where they prove entropic stability for a class of deformation of Anosov automorphisms. Although it is not written explicitly, their method can be applied for all derived from Anosov diffeomorphisms to obtain isomorphism between $f$ and its linearization for high entropy measures. But we emphasize that we are working with volume measure which has not necessarily high entropy needed to use similar arguments as them.

\subsection{Technical considerations}\label{sec:considerations.proof}

To obtain the Bernoulli property for systems with a hyperbolic structure, some features of hyperbolic systems are fundamental. These are: the existence of a pair of foliations $\mathcal F^s$, $\mathcal F^u$ which are respectively contracting and expanding, transversal to each other and absolutely continuous. With this structure in hand, the standard procedure is to take a partition (usually a partition where each element has piecewise-smooth boundaries) and prove that it is Very Weak Bernoulli (see Definition \ref{defi:VWB}). As a consequence of Ornstein Theory it follows that the system is Bernoulli. In the literature, the Bernoulli property is also obtained via Ornstein Theory arguments when we have some type of symbolic dynamics associated to the dynamical system under consideration. This is the case when we have Markov partitions or some type of Markov structure for the dynamical system. It was using this type of approach that M. Ratner \cite{Ratner1} proved that Anosov flows with $u$-Gibbs measures are Bernoulli. Also using a symbolic approach, F. Ledrappier, Y. Lima and O. Sarig \cite{LLS} proved that, with respect to an ergodic measure of maximum entropy, smooth flows with positive speed and positive topological entropy on a compact smooth three dimensional manifold are either Bernoulli or isomorphic to the product of a Bernoulli flow and a rotational flow. 

The fundamental difficulty in the partially hyperbolic context is that we lack hyperbolic behavior on the center direction and we do not have, a priori, any kind of Markov partitions (thus we lack a symbolic representation for the system). When we restrict to the context of derived from Anosov diffeomorphisms (see \ref{defi:DA}) we have the advantage that the center foliation in some sense carries some information from the center foliation of its linearization. If the linearization of a derived from Anosov diffeomorphism $f$ has negative center exponent for example, then the center foliation of $f$ has globally the same behavior (expansion or contraction) as the center foliation for the linearization.

In view of this fact, the idea to tackle the problem for derived from Anosov diffeomorphisms is to treat the center foliation as a contracting (or expanding) foliation in ``as many points as possible''. This idea is  the key to proof of Theorem \ref{theo:main}. Given a derived from Anosov diffeomorphism $f$ with linearization $A$, we have a semi-conjugacy $h$ between $f$ and $A$ (for details see \S \ref{sec:semi-conjugacy}). Assume that the center Lyapunov exponent of $A$ is negative. We prove (see Lemma \ref{lemma: distancia}) that if a pair of points $(a,b) \in M \times M$ is such that $h(a) \ne h(b)$ and $b\in \mathcal F^{cs}(a)$ then their orbits by $f$ behaves, for most of the time, as if they were in a same contracting foliation, that is, the distance between $f^n(a)$ and $f^n(b)$ is very small for most of the natural numbers $n$. Thus, if we restrict our analysis to the set of points $x\in M$ for which $h(y) \ne h(x)$ for all $y\in M\setminus \{x\}$, we can ``treat'' the center foliation essentially as if it was a contracting foliation. Although, we have no reason to assume at first that this set has total measure.

Therefore the question is: how large is the set of points which are in the injectivity domain of $h$, that is, how large is the set
\[\{ x\in M : h(y) \ne h(x) \text{ for all } y \in M\setminus \{x\} \} ?\]

 We prove in Theorem \ref{theo:main.B} that if this set has zero measure then there exists a full measure set intersecting almost every center leaf in exactly one point. We call such set as the set of atoms. Since points in two separate center leaves have distinct images by $h$ (see \S \ref{sec:semi-conjugacy}) then we can restrict our analysis to the set of atoms and again treat, in some sense, the center foliation as a ``contracting foliation".
 
Several technical issues appears when we execute the idea outlined above. One of the technical issues is that when dealing with the injectivity domain of $h$, we need to prove that we are making measurable choices of sets. For example we need to prove that the partition by collapsed pieces (see \S \ref{sec:semi-conjugacy}) is indeed a measurable partition. Later we need to apply that $h$ maps the union of collapsed pieces to a measurable set (see Lemma \ref{lemma:h(c).measurable}) and use it, together with a Measurable Choice Theorem (Theorem \ref{theo:MCT}), to prove that the extreme bottom points of collapsed arcs $c(x)$ around $x$ varies measurably on $x \in M$. 

\subsection{Generalization to higher dimensions} \label{subsection.generalization}
Our results can be easily generalized to higher dimensional situations as long as we assume some technical but usual hypothesis. We briefly explain below what are these properties and their importance. The properties are:
\begin{enumerate}
\item One dimensionality of center direction;
\item Quasi-isometric property for the lifted strong-foliations ${\mathcal F}^u$ and ${\mathcal F}^s$.
\end{enumerate}
Here a foliation $\mathcal{F}$ is quasi-isometric if, after lifting to the universal cover, there is a constant $Q$ such that $d_{\tilde{\mathcal{F}}} (x, y) < Q d_{\mathbb{R}^n} (x, y) + Q$ for all $x, y$ on the same leaf of the lifted foliation $\tilde{\mathcal{F}}.$


 Using properties $(1)$ and $(2)$, R. Ures (see Theorem $1.2$ in \cite{Ures}) proved that the semi-conjugacy $h$ maps center leaves of $f$ to center leaves of $A$ and two points with the same image should belong to the same center manifold. The authors proved in \cite{PTV} that if the disintegration of an ergodic measure for a derived from Anosov diffeomorphism is atomic then it is mono-atomic (this is needed in the proof of Theorem \ref{theo:main.B}). Therefore, we can generalize (using the same proof) Theorems \ref{theo:main} and \ref{theo:main.B} for higher dimensional case as follows.

\begin{maintheorem}\label{theo:main.C}
 Let $f:\mathbb T^n \rightarrow \mathbb T^n$ be a $C^2$ volume preserving derived from Anosov diffeomorphism with one dimensional center foliation and quasi-isometric strong stable and unstable foliations. Let  $A:\mathbb T^n \rightarrow \mathbb T^n$ be the linearization of $f$. 
Assume  that one of the following properties occurs:
 \begin{enumerate}
 \item $\lambda^c_A <0$ and $\mathcal F^{cs}$ is absolutely continuous or,
 \item $\lambda^c_A>0$ and $\mathcal F^{cu}$ is absolutely continuous.
 \end{enumerate}
Therefore if $f$ is a Kolmogorov automorphism, then $f$ is Bernoulli.
\end{maintheorem}

\begin{maintheorem} \label{theo:main.D}
Let $f:\mathbb T^n \rightarrow \mathbb T^n$ be a $C^2$ volume $(m)$ preserving derived from Anosov diffeomorphism with one dimensional center foliation and quasi-isometric strong stable and unstable foliations. Let $h$ be the the semi conjugacy to the linear Anosov diffeomorphism. Then, $h$ is $m-$almost everywhere injective. More precisely, the following dichotomy is valid:
\begin{itemize}
\item Either the set $\mathcal C$ (where $h$ fails to be injective) has zero volume, or
\item $\mathcal C$ has full measure and $(f, m)$ is virtually hyperbolic. 
\end{itemize}
In the latter case,  $(f, m)$ is Kolmogorov.
\end{maintheorem}

\section*{Acknowledgement}

We would like to thank Federico R. Hertz, A. Katok, J. P. Thouvenot for conversations and comments. We also thank Anne Bronzi for usefull conversations about measurability on Theorem \ref{theo:main.B}. G.P. was partially supported by FAPESP (grants \#2009/16792-8, \#2015/02731-8). R.V. was partially suported by FAPESP (grant \#2011/21214-3). A. T was enjoying a one year research period  in Université Paris-Sud and supported by FAPESP (\#2014/23485-2).

\section{Preliminaries}\label{sec:preliminaries}

\subsection{Partially Hyperbolic Dynamics}

In this paper we use the following definition of partial hyperbolicity. 

\begin{definition}
Given a smooth compact Riemannian manifold $M$. A diffeomorphism $f:M \rightarrow M$ is called partially hyperbolic if the tangent bundle of the ambient manifold  admits an invariant decomposition $TM = E^s \oplus E^c \oplus E^u$, such that all unit vectors $v^{\sigma} \in E^{\sigma}_x, \sigma \in \{s, c, u\}$ for  any $x, y, z \in M$
\[
 \|D_xf v^s \| < \|D_yf v^c\| < \| D_zf v^u\|
\] and
$\|D_xf v^s\| < 1 < \|D_zf v^u\|$
where $v^s, v^c$ and $v^u$ belong respectively to $E_x^s, E^c_y$ and $E^u_z$. 
\end{definition}

This is also referred as absolute partially hyperbolic to distinguish from the pointwise definition of partial hyperbolicity which has the following substitution for the first relation above:  for any $ x \in M, \|D_xf v^s \| < \|D_xf v^c\| < \| D_xf v^u\|$.

It is well known that for partially hyperbolic diffeomorphisms, there are foliations $\mathcal F^{\tau}, \tau=s,u,$ tangent to the sub-bundles $E^{\tau}, \tau=s,u$, called \textit{stable} and \textit{unstable foliation} respectively (for more details see for example \cite{YP}). On the other hand, the integrability of the central sub-bundle $E^c$  is a subtle issue and is not the case in the general partially hyperbolic setting (see \cite{HHU}). However, by a result of M. Brin, D. Burago, S. Ivanov \cite{BBI}, all absolute partially hyperbolic diffeomorphisms on $\mathbb T^3$ admit a central foliation tangent to $E^c$. 

Let $f: \mathbb T^n \rightarrow \mathbb T^n$ be a partially hyperbolic diffeomorphism. Consider $f_* : \mathbb{Z}^n \rightarrow \mathbb{Z}^n$ the action of $f$ on the fundamental group of $\mathbb T^n$. $f_*$ can be extended to $\mathbb{R}^n$ and the extension is the lift of  a unique linear automorphism $A :\mathbb T^n \rightarrow \mathbb T^n$.

\begin{definition} \label{defi:linearizacao}
Given $f: \mathbb T^n \rightarrow \mathbb T^n$ a partially hyperbolic diffeomorphism. The unique linear automorphism $A: \mathbb T^n \rightarrow \mathbb T^n$  with lift $f_{*} : \mathbb R^n \rightarrow \mathbb R^n$ as constructed in the previous paragraph, is called the linearization of $f$.
\end{definition}

It can be proved that the linearization $A$ of a partially hyperbolic diffeomorphism $f: \mathbb{T}^3 \rightarrow{T}^3$, is a partially hyperbolic automorphism of torus (\cite{BBI}). 

A. Hammerlindl proved that any (absolutely) partially hyperbolic diffeomorphism $f$ on $\mathbb T^3$ is leaf conjugated to its linearization.  This means that there exist an homeomorphism $H \colon \mathbb{T}^3 \rightarrow \mathbb{T}^3$ such that $H$ sends the central leaves of $f$ to central leaves of $f_*$ and conjugates the dynamics of the leaf spaces. In the high dimensional case  one need to assume one dimensionalaity of center bundle  and quasi-isometric strong foliations  (\cite{H}) to conclude the same results.


\begin{definition} \label{defi:DA}
We say that $f\colon \mathbb{T}^n \rightarrow \mathbb{T}^n$ is a derived from Anosov diffeomorphism or just a DA diffeomorphism if it is partially hyperbolic and its linearization is a hyperbolic automorphism (no eigenvalue of norm one).
\end{definition}

\subsubsection{The semi-conjugacy}\label{sec:semi-conjugacy}

If $f:\mathbb T^3 \rightarrow \mathbb T^3$ is a DA diffeomorphism, then by results of J. Franks \cite{Franks} and A. Manning \cite{Manning} there is a semi-conjugacy $h: \mathbb{T}^3 \rightarrow \mathbb{T}^3$, which we will call  the Franks-Manning semi-conjugacy, between $f$ and its linearization $A$, that is,
\begin{equation} \label{semiconjugacy}
A \circ h = h \circ f
\end{equation}
Moreover, this semi-conjugacy has the property that there exists a constant $K  \in \mathbb{R}$ such that  if $\tilde{h} : \mathbb{R}^3 \rightarrow \mathbb{R}^3 $ denotes the lift of $h$ to $\mathbb{R}^3$ we have $\|\tilde h(x) - x\| \leq K$ for all $x \in \mathbb{R}^3$, and given two points $a,b \in \mathbb R^3$, there exists a constant $\Omega >0$ with
\begin{eqnarray}\label{eq:h}
\tilde{h}(a) = \tilde{h}(b) \Leftrightarrow \| \tilde{f}^n(a) - \tilde{f}^n(b)\| < \Omega , \forall n\in \mathbb Z. 
\end{eqnarray}

R. Ures \cite{Ures} proved that $h$ takes center leaves of $f$ onto center leaves of $A$, that is, 

\[\mathcal{F}^c_A (h(x)) = h(\mathcal{F}^c_f (x)).\]

Therefore, given any point $x_0 \in \mathbb T^3$ the set $h^{-1}(\{x_0\})$ is uniformly bounded connected set (i.e. a segment) inside the center foliation.

Given a point $x\in \mathbb T^3$ define the set $c(x) \subset \mathcal F^c(x)$ by: 
\[ c(x):= h^{-1}(\{h(x)\}).\]
By the above discussion, the diameter of $c(x)$ is uniformly bounded in $x$.\\
Take

\begin{eqnarray}\label{eq:C}
\mathcal C := \bigcup_{y \in \{x\in \mathbb T^3 \; | \; c(x)\neq \{x\} \}} c(y). \end{eqnarray}
It is easy to see that $f(\mathcal C) = \mathcal C$, for if $h(a)=h(b)$ then by \eqref{semiconjugacy} we have
\[h(f(a)) = A(h(a)) = A(h(b)) = h(f(b)),\]
and if $h(a) \ne h(b)$ then by \eqref{semiconjugacy} we have
\[h(f(a)) = A(h(a)) \ne A(h(b)) = h(f(b)).\]

\subsection{Measurable partitions and disintegration of measures}

Let $(M, \mu, \mathcal B)$ be a probability space, where $M$ is a compact metric space, $\mu$ a probability measure and $\mathcal B$ the Borelian $\sigma$-algebra.
Given a partition $\mathcal P$ of $M$ by measurable sets, we associate the probability space $(\mathcal P, \widetilde \mu, \widetilde{\mathcal B})$ by the following way. Let $\pi:M \rightarrow \mathcal P$ be the canonical projection, that is, $\pi$ associates to a point $x$ of $M$ the partition element of $\mathcal P$ that contains it. Then we define $\widetilde \mu := \pi_* \mu$ and $ \widetilde{\mathcal B}:= \pi_*\mathcal B$.

\begin{definition} \label{definition:conditionalmeasure}
 Given a partition $\mathcal P$. A family $\{\mu_P\}_{P \in \mathcal P} $ is a \textit{system of conditional measures} for $\mu$ (with respect to $\mathcal P$) if
\begin{itemize}
 \item[i)] given $\phi \in C^0(M)$, then $P \mapsto \int \phi \mu_P$ is measurable;
\item[ii)] $\mu_P(P)=1$ $\widetilde \mu$-a.e.;
\item[iii)] if $\phi \in C^0(M)$, then $\displaystyle{ \int_M \phi d\mu = \int_{\mathcal P}\left(\int_P \phi d\mu_P \right)d\widetilde \mu }$.
\end{itemize}
\end{definition}

When it is clear which partition we are referring to, we say that the family $\{\mu_P\}$ \textit{disintegrates} the measure $\mu$.  

\begin{proposition} \cite{EW, Ro52}
 Given a partition $\mathcal P$, if $\{\mu_P\}$ and $\{\nu_P\}$ are conditional measures that disintegrate $\mu$ on $\mathcal P$, then $\mu_P = \nu_P$ $\widetilde \mu$-a.e.
\end{proposition}

\begin{corollary} \label{cor:same.disintegration}
 If $T:M \rightarrow M$ preserves a probability $\mu$ and the partition $\mathcal P$, then  $T_*\mu_P = \mu_{T(P)}, \widetilde \mu$-a.e.
\end{corollary}
\begin{proof}
 It follows from the fact that $\{T_*\mu_P\}_{P \in \mathcal P}$ is also a disintegration of $\mu$.
\end{proof}

\begin{definition} \label{def:mensuravel}
We say that a partition $\mathcal P$ is measurable (or countably generated) with respect to $\mu$ if there exist a measurable family $\{A_i\}_{i \in \mathbb N}$ and a measurable set $F$ of full measure such that 
if $B \in \mathcal P$, then there exists a sequence $\{B_i\}$, where $B_i \in \{A_i, A_i^c \}$ such that $B \cap F = \bigcap_i B_i \cap F$.
\end{definition}

\begin{theorem}[Rokhlin's disintegration \cite{Ro52}] \label{teo:rokhlin} 
 Let $\mathcal P$ be a measurable partition of a compact metric space $M$ and $\mu$ a Borelian probability. Then there exists a disintegration by conditional measures for $\mu$.
\end{theorem}
In general the partition by the leaves of a foliation may be non-measurable. It is for instance the case for the stable and unstable foliations of a linear Anosov diffeomorphism. Therefore, by disintegration of a measure along the leaves of a foliation we mean the disintegration on compact foliated boxes. In principle, the conditional measures depend on the foliated boxes, however, two different foliated boxes induce proportional conditional measures. See \cite{AVW} for a discussion. We define absolute continuity of foliations as follows:  

\begin{definition} \label{ABS}
We say that a foliation $\mathcal F$ is absolutely continuous if for any foliated box, the disintegration of volume on the segment leaves have conditional measures equivalent to the Lebesgue measure on the leaf.
\end{definition}

\begin{definition}\label{defi:atomica}
We say that a foliation $\mathcal F$ has atomic disintegration with respect to a measure $\mu$, or that $\mu$ has atomic disintegration along $\mathcal F$, if for any foliated box, the disintegration of $\mu$ on the segment leaves have conditional measures which are finite sums of Dirac measures.
\end{definition}

For the stable and unstable foliations of a partially hyperbolic diffeomorphism, it is well known (see \cite{BP}, \cite{YP}, \cite{BrinPesin}, \cite{PuShu1}, \cite{PuShu2}) that they are absolutely continuous and furthermore, the holonomy map between two transversals is absolutely continuous with Jacobian going to one as the transverses approach each other.

In order to prove Theorem \ref{theo:main.B} we need to use, what is known as the Measurable Choice Theorem. This result has an intrinsic interest from the mathematical point of view, although it has appeared in the context of Decision Theory from Economics. It has been proved by R. J. Aumann \cite{Aumann}: 

\begin{theorem}[Measurable Choice Theorem, \cite{Aumann}] \label{theo:MCT}
Let $(T,\mu)$ be a $\sigma$-finite measure space, let $S$ be a Lebesgue space, and let $G$ be a measurable subset of $T\times S$ whose projection on $T$ is all of $T$. Then there is a measurable function $g:T\rightarrow S$, such that $(t,g(t)) \in G$ for almost all $t \in T$.
\end{theorem}

\subsection{The Bernoulli property}
A measure preserving transformation has the Bernoulli property or  is called Bernoulli if it is isomorphic to a Bernoulli shift. In fact a dynamical system is Bernoulli if and only if there exists a finite partition which is independent and generating.

\subsection{The Kolmogorov property}

For two partitions $\alpha$ and $\beta$ of a measurable space $X$, the union of $\alpha$ and $\beta$ is the partition defined by
\[\alpha \vee \beta:= \{ A\cap B: A\in \alpha, \beta \in \beta\}.\]
For a finite number of partitions $\alpha_1,...,\alpha_N$ of $X$, $\bigvee_{i=1}^N \alpha_i$ denotes their union as defined above. If we have an infinite sequence of partitions $\{\alpha_i\}_{i}$, then $\bigvee_{-\infty}^{+\infty}\alpha_i$ denotes the smallest $\sigma$-algebra with respect to which for every $i$, each element of $\alpha_i$ is a measurable set. 
The notations of finite union and infinite unions are similar but they have different meanings: one is a partition and the other is a $\sigma$-algebra (this is a standard notation on the literature).\\

The next definition of a Kolmogorov automorphism is suitable for our purposes, although it is not the standard one it is easily equivalent to it by the Martingale Convergence Theorem as remarked by D. Ornstein (\cite{OrnsteinBook} Part III).

\begin{definition}
Let $f$ to be an automorphism of a Lebesgue space $(Y,\mu)$ and $\xi$ a finite partition of $Y$. We say that $\xi$ is a Kolmogorov partition (or a $K$-partition) if for any $B \in \bigvee_{-\infty}^{+\infty} f^{-k}\xi$, given $\varepsilon >0$ there is an $N_0 = N(\varepsilon, B)$ such that for all $N' \geq N \geq N_0$ and $\varepsilon$-almost every atom $A$ of $\bigvee_{N}^{N'}f^k\xi$ we have
\[ \left| \frac{\mu(A \cap B)}{\mu(A)} - \mu(B) \right| \leq \varepsilon.  \]

If the automorphism $f$ has a generating partition $\xi$ which is Kolmogorov then we say that $f$ is a Kolmogorov automorphism (or a $K$-automorphism).
\end{definition}

The following lemma, which follows from works of Pinsker, Rokhlin-Sinai, states that if $f$ is Kolmogorov then every finite partition is a Kolmogorov partition.

\begin{lemma} [Pinsker \cite{Pinsker}, Rokhlin-Sinai \cite{RS}] \label{lema:Kolmogorov}
Let $f$ be a $K$-automorphism of a Lebesgue space $(Y,\mu)$, then every finite partition  $\xi$ of $Y$ is a Kolmogorov partition.
\end{lemma}


\subsection{Kolmogorov property and entropy}
It is well known that the Kolmogorov property can be characterized using entropy. The following Theorem of V. Rokhlin and Y. Sinai \cite{RS} shows that Kolmogorov property is an enough and sufficient condition to have uncertainty of a system in the sense that every non-trivial finite partition generates a positive information.

\begin{definition}
Let $(X,\mu)$ be a Lebesgue space and $T: (X,\mu) \rightarrow (X,\mu)$ a measure preserving automorphism. The family $\pi$ defined by
\[\pi = \{ A \subset X : h_{\mu}(T,\alpha) = 0 , \alpha:=\{A,X-A\} \}\]
is a $\sigma$-algebra called the Pinsker $\sigma$-algebra of $T$.
\end{definition}

\begin{theorem} [Rokhlin-Sinai, \cite{RS}] \label{teo:RS}
Let $(X,\mu)$ be a Lebesgue space (i.e, isomorphic to $([0,1], Leb)$ ). An automorphism $f: (X,\mu) \rightarrow (X,\mu)$ is Kolmogorov if, and only if, it has completely positive entropy, that is, given any non-trivial finite partition $\alpha$ of  $X$ we have
\[h_{\mu}(f,\alpha) >0.\]

That is, $f$ is Kolmogorov if, and only if, the Pinsker $\sigma$-algebra $\pi$ of $f$ is trivial.
\end{theorem}

\subsection{Kolmogorov property and partial hyperbolicity}

In the light of Rokhlin-Sinai's Theorem, if one can find a good characterization of the Pinsker $\sigma$-algebra then we can verify the Kolmogorov property. Indeed, such characterization was made by M. Brin and Y. Pesin \cite{BrinPesin}, where they showed that for a $C^{1+\alpha}$ partially hyperbolic diffeomorphism with smooth measure and non-zero Lyapunov exponent $f:M \rightarrow M$ ($M$ a compact Riemannian manifold), every element of the Pinsker $\sigma$-algebra is bi-essentially saturated 
Later, F. Ledrappier and L. Young \cite{LY1} extended this characterization for $C^2$ diffeomorphisms preserving a Borel probability measure. 

\begin{theorem}[Burns-Wilkinson \cite{BW}] \label{BW}
Let $f$ be $C^2$, volume-preserving, partially hyperbolic and center bunched. If $f$ is essentially accessible, then $f$ is ergodic, and in fact has the Kolmogorov property.
\end{theorem}

For the one-dimensional center case, a stronger type of center bunching hypothesis is trivially satisfied and the $C^2$ condition can be weakened to $C^{1+\alpha}$ (see \cite{BW} for details).

\begin{corollary}[Corollary $0.4$ of \cite{BW}]\label{co:BW}
Let $f$ be a $C^{1+\alpha}$, volume preserving and partially hyperbolic with $\dim (E^c) = 1$. If $f$ is essentially accessible, then $f$ is ergodic, and in fact has the Kolmogorov property.
\end{corollary}

\begin{proposition}
Let $M$ be a compact Riemannian manifold and $f: M \rightarrow  M$ a volume preserving $C^{1+\alpha}$ partially hyperbolic diffeomorphism and one-dimensional center direction. Thus $f$ is Kolmogorov if, and only if, it is essentially accessible.
\end{proposition} 
\begin{proof}
By \cite{YP3} we know that Kolmogorov property implies that every $u$-saturated set has measure zero or one. Without loss of generality assume that the center Lyapunov exponent is non-positive. Since the union of accessibility classes are in particular $u$-saturated they must have zero or full measure, thus it is essentially accessible. The other way is a direct consequence of Corollary \ref{co:BW}. \end{proof}

By a result of F. Hertz, M. A. Hertz and R. Ures \cite{HHU1} for the one dimensional center case, the accessibility hypothesis (consequently the Kolmogorov hypothesis) is $C^r$ open and dense in the set of volume preserving diffeomorphisms, for any $1 \leq r \leq \infty$.


\subsection{The Very Weak Benoulli property} \label{section:VWB}

Let $(X,\nu)$ be a Lebesgue space. We will always assume that the measures we consider are normalized, that is, if $A\subset X$ is a measurable subset then by the measure on $A$ we mean the measure $\nu(A)^{-1}\nu$.

\begin{definition}
Let $\xi$ be a finite partition of $X$. We say that a property holds for $\varepsilon$-almost every element of $\xi$ if the union of those elements for which the property fails has $\nu$-measure at most $\varepsilon$.
\end{definition}

In order to define what is called a Very Weak Bernoulli (VWB) partition we need to introduce a metric in the space of finite partitions. 

Let $\xi_i = \{A_1^{(i)}, ... ,A_m^{(i)} \}$ and $\eta_i = \{B_1^{(i)}, ... ,B_m^{(i)}\} , 1\leq i \leq n$, be two sequences of partitions of Lebesgue spaces $(X,\nu)$ and $(Y,\mu)$ respectively. We write
\[ \{\xi_i\}_1^n \sim \{\eta_i \}_1^n \]

if for any $1 \leq k_i \leq m$ and $1 \leq i \leq n$,
\[ \nu \left( \bigcap_{i=1}^{n} A_{k_i}^{(i)} \right) = \mu \left( \bigcap_{i=1}^{n} B_{k_i}^{(i)} \right). \]

\begin{definition}
We say that the distance between $\{\xi_i\}_1^n$ and $\{\eta_i \}_1^n$ is less or equal $\varepsilon$, and denote it by
\[\bar{d}(\{\xi_i\}_1^n,\{\eta_i \}_1^n) \leq \varepsilon \]
if there exist sequences of partitions $\bar{\xi}_i = \{ \bar{A}^{(i)}_1, ... , \bar{A}^{(i)}_m \}$ and $\bar{\eta}_i = \{\bar{B}_1^{(i)}, ... ,\bar{B}_m^{(i)} \}$ of a Lebesgue space $(Z,\lambda)$ such that
\[ \{\xi_i\}_1^n \sim  \{\bar{\xi}_i\}_1^n  \text{ and }  \{ \eta_i\}_1^n \sim \{\bar{\eta}_i\}_1^n,  \]
and
\[\frac{1}{n}\sum_{i=1}^{n} \sum_{j=1}^{m} \lambda( \bar{A}_j^{(i)}  \triangle \bar{B}_j^{(i)} ) \leq \varepsilon.\]

\end{definition}

\begin{definition} \label{defi:VWB}
A partition $\xi$ of a Lebesgue space $(X,\nu)$ is called $\operatorname{VWB}$ if for any $\varepsilon>0$ there exists $N_0 =N_0(\varepsilon)$ such that for any $N' > N \geq N_0, n \geq 0$, and $\varepsilon$-almost every element $A \in \bigvee_{k=N}^{N'} f^k\xi$, we have 
\[\bar{d}(\{f^{-i}\xi\}_1^n , \{f^{-i}\xi|A \}_1^n) \leq \varepsilon . \]

Let us emphasize that the partition $\xi|A$ is considered with respect to the normalized measure.
\end{definition}

D. Ornstein proved that the $\operatorname{VWB}$ property can be used to obtain Bernoulli property for a system. 

\begin{theorem}\cite{O70} \label{O70}
If $\alpha$ is a $\operatorname{VWB}$ partition, then the system $(X, \bigvee_{-\infty}^{+\infty} f^{-i}\alpha, \mu, f)$ is Bernoulli.
\end{theorem}

\begin{theorem}[see \cite{OW}, \cite{O702}] \label{O702}
Let  $\alpha_1 \leq \alpha_2 \leq \alpha_3 \leq ...$ to be a increasing sequence of partitions of a Lebesgue space $(X, \mathcal A, \mu)$ such that 
\[ \bigvee_{i=1}^{+\infty} \bigvee_{n=-\infty}^{+\infty} f^{-n}\alpha_i   = \mathcal A\]
and, for each $i$, $(X, \bigvee_{n=-\infty}^{+\infty} f^{-n}\alpha_i, \mu, f)$ is a Bernoulli system. Then, $(X, \mathcal A, \mu, f)$ is a Bernoulli system.
\end{theorem}


Given a sequence of partitions $\{\xi_i\}_1^n$ of a Lebesgue space $(X,\nu)$, we define the sequence of integer-valued functions $l_i(x)$ by the condition $x \in A_{l_i(x)}^{(i)}$. We call the sequence of functions $l_i(x)$ the \textit{name} of the sequence of partitions $\{\xi_i\}_1^n$. \\

\begin{definition}
We say that a transformation $\theta: X \rightarrow Y$ of Lebesgue spaces $(X,\nu)$ and $(Y,\mu)$ is $\varepsilon$-measure preserving if there exists a set $E \subset X$ such that $\nu(E) \leq \varepsilon$ and for every measurable set $A \subset X \setminus E$,
\[ \left|\frac{\nu(A)}{\mu(\theta(A))} - 1 \right| \leq \varepsilon. \]

\end{definition}

Define the function $e: \mathbb N \rightarrow \{0,1 \}$ by $e(0)=0$ and $e(n)=1$ for all $n>0$. The following lemma makes a geometric approach possible to prove the Very Weak Bernoulli property. \\ 

\begin{lemma}[ Lemma $4.3$ in \cite{CH}] \label{lema: bernoulli}
Let $\{\xi_i\}_{1}^{n}$ and $\{\eta_i\}_{1}^{n}$ be two sequences of partitions of Lebesgue spaces $(X,\nu)$ and $(Y,\mu)$ with name functions $l_i(x)$ and $m_i(x)$ respectively. Assume that there exists an $\varepsilon$-measure preserving transformation $\theta:X \rightarrow Y$, satisfying
\[\frac{1}{n} \sum_{i=1}^{n}e(l_i(x) - m_i(\theta(x))) \leq \varepsilon \]
\noindent for all $x\in X$ with possible exception for a set $E \subset X$ with $\nu(E) \leq \varepsilon$. Then
\[\bar{d}(\{\xi_i\}_{1}^{n} , \{\eta_i\}_{1}^{n}) \leq 16 \varepsilon.\]
\end{lemma}

The proof of Theorem \ref{theo:main} is based on the construction of a function $\theta$ satisfying the previous lemma. The approach used in \cite{OW} and \cite{YP3} was to construct $\theta$ preserving stable manifolds, so it guarantees that the $f$-orbits of $x$ and $\theta(x)$ cannot be far from each other. In our case, due to the presence of the center manifold, it is not possible to construct $\theta$ preserving the stable manifold, thus we have to work a little bit with the semi-conjugacy to obtain good non-divergence along the center for at least large sets.

\section{Proof of Theorem \ref{theo:main.B}}

\begin{proof}[Proof of Theorem B]

We use the notation $c(x)$ and $\mathcal C$ as introduced on \S \ref{sec:semi-conjugacy}.
First of all let us mention that the dichotomy in the Theorem B immediately implies that $h$ is almost everywhere injective. Indeed, if $m(\mathcal{C}) =0$ then $h$ is injective outside $\mathcal{C}$ which has full measure, and if $m(\mathcal{C})= 1$ then we prove that $m$ is virtually hyperbolic or in other words there exists a full measurable subset intersecting central leaves in at most one point and consequently $h$ restricted to this subset is injective.

So let us prove the dichotomy. If the Lebesgue measure of $\mathcal{C}$ does not vanish, then we claim that $f$ is accessible, otherwise by \cite{HU} it would be topologically conjugate to its linearization, hence the set $\mathcal C=\emptyset$, which is an absurd. Thus $f$ is accessible, therefore Kolmogorov. In particular $f$ and $f^2$ are ergodic, then we can assume without loss of generality that $f$ preserves the orientation of $\mathcal F^c$ (otherwise we work with $f^2$). Since $\mathcal C$ is an invariant set with positive volume, it has full volume.

Now we need to prove that $m$ is virtually hyperbolic. To do this, we restrict $m$ to the full measure set $\mathcal{C}$ and consider its disintegration along the partition by intervals $c(x), x \in \mathcal{C}:$

\begin{lemma} \label{lemma: measurable}
If $m(\mathcal C) = 1$ then the partition $\mathscr C =  \{c(x) : x \in \mathcal C\}$ of $\mathbb T^3$ by collapsed arcs of center manifolds is a measurable partition. 
\end{lemma}
\begin{proof}
Consider $\{B_i\}_{i \in \mathbb N}$ a countable basis of open sets on $\mathbb T^3$. Then, we know that each point $x_0 \in \mathbb T^3$ can be written as an intersection
\[\{x_0\} = \bigcap B_i^*,\]
where $B_i^* = B_i$ or $B_i^c$. Thus,
\[h^{-1}(\{x_0\}) = \bigcap h^{-1}(B_i^*).\]

Since $h$ is continuous, $h^{-1}(B_i^*)$ is Lebesgue measurable (since it is an open or a closed set). The countable family $\{h^{-1}(B_i^*)\}$ separates the sets $h^{-1}(\{x_0\})$. Thus $\mathcal C$ is a measurable partition.
\end{proof}

Let $\{m_{c(x)}\}_{c(x)\in \mathcal C}$ be the Rokhlin disintegration of volume $m$ on the partition $\mathscr{C}$.

Define $\eta:=h_*m$, notice that $\eta$ is an invariant measure for $A$ (recall that $A$ is the linearization of $f$). And because the collapsed intervals have full measure by hypothesis, then $\eta$ is an atomic and invariant measure for the linear map $A$. We apply Theorem B of \cite{PTV} to $\eta$ and obtain that $\eta$ has atomic disintegration and one atom per leaf. Although Theorem B of \cite{PTV} deals with volume measure we point out that the same argument works \textit{ipsis litteris} by changing volume to $\eta$.
 
One atom per center leaf of the disintegration of $\eta$ implies that for $m$ almost every $x$ the disintegration $\{m_{c(x)}\}_{c(x)\in \mathcal C}$ satisfies: if $c(x) \neq c(y)$ then $c(x)$ and $c(y)$ are in distinct center leaves. 

Observe that up to now, we get a full Lebesgue measurable subset which intersects almost all leaves at most in a unique interval (a collapse interval $c(x)$). A priori the conditional measure of $m$ supported on $c(x)$ may be non-atomic. In what follows we prove that this is not the case and indeed such conditional measures are Dirac measures. This implies virtual hyperbolicity of $m.$

\subsection{Virtual hyperbolicity of Lebesgue measure}
 
 In this subsection we prove that $m$ disintegrates into (mono-)atomic measures along central foliation. Let us mention that along the proof we use just ergodicity of $m$. 
 
First of all we prove  measurability of the set $h(\mathcal C)$. We mention that Ures \cite{Ures}  observed that $m(h(\mathcal C))=0$. However, a proof for  measurability of $h(\mathcal{C})$ is necessary. Since this is not obvious, we give a proof below.

\begin{lemma}\label{lemma:h(c).measurable}
The set $$h(\mathcal C)=\{x \in \mathbb T^3: \#\{h^{-1}(x)\} \ne 1 \}$$ is a measurable set, in fact a Souslin set. Where $\#\{h^{-1}(x)\}$ means the cardinality of the set $\{h^{-1}(x)\}$.
\end{lemma}
\begin{proof}

Let us denote $\tilde f$ and $\tilde h$, respectively, as the lift of $f$ and $h$ to the universal cover $\mathbb R^3$. As pointed out in \S \ref{sec:semi-conjugacy} 
$\tilde h(x)=\tilde h(y)$ if and only if $||\tilde f^n(x)-\tilde f^n(y)|| \leq \Omega, \; \forall n \in \mathbb Z$, where $\Omega$ is some constant which only depends on $f$. Hence, we may characterize the points where $\tilde h$ is injective as the following:

$$\#\{\tilde h^{-1}(\{x\})\}=1 \text{ iff }
x= \bigcap_{n \in \mathbb Z}\tilde f^{-n}(B(\tilde f^n(x),\Omega)),$$
where $B({\tilde f}^n(x),\Omega)) \subset \mathbb R^3$ is the ball centered at $\tilde f^n(x)$ and radius $\Omega$. 
Define $B_{\tilde f}(x):=\bigcap_{n\in \mathbb Z}\tilde f^{-n}(B(\tilde f^n(x),\Omega))$ and $B_{\tilde f}^n(x):=\bigcap_{k=-n}^n{\tilde f}^{-k}(B({\tilde f}^k(x),\Omega))$

Consider the set $P:= \{(x,y) \in \mathbb R^3 \times \mathbb R^3 \; | \; x \in B_{\tilde f}(y)\}$, also notice that $P=\{ (x,y) \; | \; \tilde{h}(x)=\tilde{h}(y) \}$, and define $P_n:= \{(x,y) \in \mathbb R^3 \times \mathbb R^3 \; | \; x \in B^n_{\tilde f}(y)\}$. It is trivial to see that the diagonal $\Delta \subset P$, since $\tilde h(x)=\tilde h(x)$. Hence $P \setminus \Delta$ contains the information of the points for which $\tilde h$ is not injective, more precisely: $\pi_1(P)= \widetilde{\mathcal C}$, where $\pi_1:\mathbb R^3 \times \mathbb R^3 \rightarrow \mathbb R^3$ is the projection onto the first coordinate and $\widetilde{\mathcal C}$ is the lift of $\mathcal C$ to the universal cover.

Because $\tilde f$ is a homeomorphism, then $P_n$ is a Borel set. And because $P=\bigcap_{n \in \mathbb N} P_n$, then $P$ is a Borel set. Therefore, $\tilde h(\widetilde{\mathcal C})=\tilde h \circ \pi_1(P)$ is a Souslin set by \cite[Corollary 1.10.9]{bogachevI}. Now projecting everything to the torus and because the projection is a local homeomorphism we obtain that $h(\mathcal C)$ is a Souslin set, in particular a measurable set.

\end{proof}

Consider $\phi(t,.):\mathbb T^3 \rightarrow \mathbb T^3$ the flow on $\mathbb T^3$ having constant speed one in the center direction for the linearization $A$. More precisely, we know that the leaves of the center foliation of $A$ are straight lines and orientable by assumption. Define $\phi(t,x)$ the unique  point in the $\mathcal F^c_A(x)$ which has distance $t$ inside this center leaf and in the positive direction from $x$.

In particular, the above lemma implies that  $$ \phi(-1/n, h(\mathcal C))$$ is a measurable set. Consider the set $h^{-1}(\phi(-1/n, h(\mathcal C)))$ which is a measurable set since $h$ is continuous. 
 
 Now consider $\Sigma := \mathcal C / \sim $ where the relation is given by: $x \sim y$ iff $x \in c(y)$. 
 
Let us define a function
$$ \psi_n:\Sigma \rightarrow h^{-1}(\phi(-1/n, h(\mathcal C))) \subset \mathbb T^3 $$
be a function given by the Measurable Choice Theorem \ref{theo:MCT} applied to the product $\Sigma\times \mathbb T^3$ and the measurable set we consider is $h^{-1}(\phi(-1/n, h(\mathcal C)))$. We have to use this theorem because we cannot assure that $h^{-1}(\phi(-1/n, h(\mathcal C)))$ intersects the center leaf in one point, it could, of course, intersect on a segment.
 
 But notice that fixed $[x] \in \Sigma$ then $\psi_n([x])$ is an increasing sequence (recall that we have an orientation for the center direction). Therefore we can define the following function
 \begin{eqnarray*}
\psi : \Sigma &\rightarrow & \mathbb T^3\\
\left[x\right] & \mapsto & \lim_{n \rightarrow \infty} \psi_n([x])
 \end{eqnarray*}
and by its construction $\psi([x])$ is the lower extreme of $c(x)$. We now want to prove that the image $Im(\psi)$ is a measurable set. Notice that $\psi$ is measurable function because it is the limit of measurable functions.

By Lusin's theorem there exists a compact set $K_n \subset \Sigma$ such that $\psi|K_n$ is a continuous function and $\hat m(\Sigma \setminus K_n) < 1/n$ where $\hat m :=\pi_* m$ and $\pi:\mathcal C \rightarrow \Sigma$ is the projection. Because $\psi|K_n$ is a continuous function $\psi(K_n)$ is a compact set. Now $\hat m(\Sigma)( \bigcup_{n}\psi(K_n))=\hat m(\Sigma)$, without loss of generality we may consider that $\Sigma = \bigcup_{n}\psi(K_n)$, therefore
$$ \psi(\Sigma) = \bigcup_{n}\psi(K_n)  \text{ is a measurable set.} $$

We have proven so far that the base of the intervals from $\mathcal C$ forms a measurable set and we call these sets as point zero, that is if $x \in \mathcal C$ then $0_x$ means the base point associated to the segment $c(x)$. Let us denote the set of these `zero" points as $\Sigma_0$. Observe that $\Sigma_0$ may be identified with $\Sigma$ and equipped with the quotient measure $\hat{m}.$ If $y \in c(x)$ then $[0_x,y]$ stands for the segment inside the center leaf which contains $0_x$ and $y$.

 Let $W_{\epsilon} := \{\phi(t,0_y) \; | \; 0 \leq t \leq \epsilon, \; y \in \mathcal C \}$. Since $[0,\epsilon]\times \Sigma_0$ is a Souslin set and the flow $\phi$ is a continuous map, the set $W_{\epsilon}$ is a measurable set by \cite[Proposition 1.10.8]{bogachevI}. Define
\begin{eqnarray*}
 \mu_{\epsilon} : \Sigma_0 &\rightarrow& [0, 1] \subset \mathbb R\\
\;  [y] &\mapsto& m_y(W_{\epsilon}),
\end{eqnarray*}
this is a measurable function by Rokhlin's Theorem \ref{teo:rokhlin}. Hence, \[\mu_{\epsilon}^{-1}([0,\alpha]) = \{0_y \; | \; m_y(W_{\epsilon}) \leq \alpha \}\] is a measurable set.\\
We now consider the only three possible cases:
 
 \begin{enumerate}

 \item  for all $\epsilon >0$ and all $\alpha \in (0,1]$,  $ \widehat m(\mu_{\epsilon}^{-1}((0,\alpha]))=0$;
  \item  there exists $\epsilon >0$ and there exists $\alpha \in (0,1)$ such that $\widehat m(\mu_{\epsilon}^{-1}((0,\alpha]))>0$;
  
  \item  there exists $\epsilon>0$ such that $\widehat m(\mu_{\epsilon}^{-1}((0,1]))>0$ and  $\widehat m(\mu_{\epsilon}^{-1}((0,\alpha]))=0$ for all $\alpha<1$.

\end{enumerate}

Assume we are in the first case. Observe that since the elements $c(x) \in \mathscr C$ have uniformly bounded length, for large enough $\epsilon>0$  we have $ \widehat m(\mu_{\epsilon}^{-1}([0,1]))=1$. Thus, for a large enough $\epsilon>0$ we must have
\[\widehat m(\mu_{\epsilon}^{-1}(\{0\}))=1.\]
But this contradicts the fact that $m(\mathcal C)=1$. Thus the first case can not occur.\\

Now, assume we are in the second case. By the definition of $\mu_{\epsilon}$ we have that
 \[0<\widehat m(\mu_{\epsilon}^{-1}((0,\alpha])) \Rightarrow m(Q(\epsilon,\alpha)) 
  \in (0,\alpha],\]
  where $Q(\epsilon,\alpha):=\phi([0,\epsilon]\times \mu_{\epsilon}^{-1}((0,\alpha]))$.
 Take the union of the iterates of $Q(\epsilon, \alpha)$, that is
 \[ Q:=\bigcup_{i \in \mathbb Z} f^{i}( Q(\epsilon,\alpha) ),\]
 which is an $f$-invariant set.
 Since $f_{*}m_y = m_{f(y)}$ then, for each $y \in Q$ we have 
 \[c(y) \cap Q \subset \{z: m_{y}([0_y,z]) \leq \alpha \},\]
 which implies $0<m(Q)\leq \alpha <1$. This contradicts the ergodicity of $f$, thus the second case can not occur as well.\\

 At last assume we are in the third case.  In this case we have
 \[\bigcup_{n\in \mathbb N}\widehat m(\mu_{\epsilon}^{-1}((0,1-1/n]))=0 \Rightarrow \widehat m(\mu_{\epsilon}^{-1}(\{1\}))=1.\]
 Let $\epsilon_0 := \inf \{\epsilon :   \widehat m(\mu_{\epsilon}^{-1}(\{1\}))=1 \}.$ Then, $ \widehat m(\mu_{\epsilon_0+1/n}^{-1}(\{1\}))=1$ for any $n \in \mathbb N$, which implies

\[m(Q(n)) = 1,\]
where $Q(n):=\phi([0,\epsilon_0+1/n]\times \mu_{\epsilon_0+1/n}^{-1}(\{1\})  )$. Take $Q = \bigcap Q(n)$. We must have
\[m(Q) = 1.\]

But by the definition of $Q(n)$ and $Q$ we have $Q = \phi([0,\epsilon_0]\times \mu_{\epsilon_0}^{-1}(\{1\}))$. Thus, by definition of $\epsilon_0$ we have
$m(\phi(\epsilon_0,\mu_{\epsilon_0}^{-1}(\{1\}))) = m(Q) = 1,$
that is, the set of points $\Theta=\phi(\epsilon_0,\mu_{\epsilon_0}^{-1}(\{1\}))$ is a set of atoms as we wanted to show.

Thus, the disintegration is indeed atomic, and there exists exactly one atom per leaf. 

\end{proof}

\section{Proof of Theorem \ref{theo:main}}
Let $f\colon \mathbb T^3 \rightarrow \mathbb T^3$ be a $C^{2}$ volume preserving Kolmogorov partially hyperbolic diffeomorphism homotopic to a linear Anosov diffeomorphism $A\colon \mathbb T^3 \rightarrow \mathbb T^3$, as in the hypothesis of Theorem \ref{theo:main}.

Without loss of generality we can assume that the center Lyapunov exponent $\lambda^c_A$ of $A$ is negative and $\mathcal F_f^{cs}$ is absolutely continuous: $\lambda^c_A<0$ (otherwise we work with $f^{-1}$ which is homotopic to $A^{-1}$).

We use the notation $c(x)$ and $\mathcal C$ as introduced on \S \ref{sec:semi-conjugacy}. From the ergodicity of $f$ it follows that $\mathcal C$ has either full or zero volume.

In our case the existence of a center foliation requires a fine comprehension of the disintegration of volume on the center foliation. The Theorem \ref{theo:main.B} is the key to overcome this issue.

\subsection{Partition by rectangles}

Take $\mathcal E$ to be the partition of $\mathbb T^3$ by points. Let $\alpha$ be a partition of $\mathbb T^3$ by  measurable sets such that the boundary of any element in $\alpha$ is piecewise smooth and such that each atom $D \in \alpha$ is an open set with boundary of zero measure. It is easy to construct such a partition on the $3$-torus. We will prove that $\alpha$ is VWB. Then we will take a sequence of such partitions $\alpha_n$ with: $\alpha_1 \leq \alpha_2 \leq ... $ such that $\alpha_n \rightarrow \mathcal E$, concluding that $f$ is indeed Bernoulli by Theorems \ref{O70} and \ref{O702}.\\

In order to do so, 
we shall consider some specific partitions with dynamical meaning which will help us to control the behavior of $f^i\alpha$, $i \in \mathbb N$.\\

Given two points $y$ and $z$ close enough to each other, we know that $\mathcal F^{cs}(y)$ and $\mathcal F^u(z)$ will intersect each other and that the intersection is, locally, a single point. We denote this point by $[y, z]$. Sometimes along this section we also write $W^{cs}(y) \cap W^u(z)$ to mean the point $[y,z]$. 

\begin{definition}
A measurable set $\Pi$ is called a $\delta$-rectangle at a point $w$ if $\Pi \subset B(w,\delta)$ and for any $y,z \in \Pi$ the local intersection belongs to $\Pi$, that is
\[ [y,z] \in \Pi .\]
\end{definition}

The definition of rectangle here and in \cite{YP3} is basically the same, the only difference being that in Pesin's case the measure is hyperbolic, so he works with stable foliation and we replace it by center-stable foliation. Also, since Pesin deals with non-hyperbolicity, he needs to work inside regular level sets (see \cite{BP} for definition), while do not need to do such restrictions here since we are working with absolutely partially hyperbolic diffeomorphisms.\\

Note that, by the local product structure of the rectangles, we can think of a rectangle $\Pi$ as a cartesian product of $\mathcal F^u_x \cap \Pi$ and $\mathcal F^{cs}_z \cap \Pi$, where $x,z \in \Pi$.
Let $f: M \rightarrow M$ be a partially hyperbolic diffeomorphism with absolutely continuous center-stable foliation. Then we can take, for a typical $z$, $m^u_z$ the measure $m$ conditioned on $\mathcal F^u_z \cap \Pi$ and $m^{cs}_f$ the factor measure on the leaf $\mathcal F^{cs}_z$, that is, as in Definition \ref{definition:conditionalmeasure} the partition $\mathcal P = \{ \mathcal F^u_z\cap \Pi\}_{z\in \Pi}$, $\{m^u_z\}$ is a system of conditional measures and $m^{cs}_f$ is the factor measure. From the absolute continuity of the unstable and center-stable foliations it follows that for a typical $z$ the product measure 
\[m^P_R:= m^u_z \times m^{cs}_f,\]
which is defined on $\Pi$, satisfies 
\[m^P_R << m .\]

The following definition of $\varepsilon$-regular covering was given by N. Chernov and C. Haskell \cite{CH} for the case of non-uniformly hyperbolic maps and flows. The definition for partially hyperbolic diffeomorphisms is the same replacing the stable by the center-stable manifold.

\begin{definition}\label{defi:regular}
Given any $\varepsilon >0$, an $\varepsilon$-regular covering of $M$ is a finite collection of disjoint rectangles $\mathcal R=\mathcal R_{\varepsilon}$ such that:
\begin{enumerate}
\item $m(\bigcup_{R\in \mathcal R} R) > 1-\varepsilon$
\item For every $R \in \mathcal R$ we have
\[ \left| \frac{m^P_R(R)}{m(R)} -1  \right| < \varepsilon  \]

and, moreover, $R$ contains a subset, $G$, with $m(G) > (1-\varepsilon)m(R)$ which has the property that for all points in $G$,
\[\left|  \frac{dm^P_R}{dm} -1  \right| < \varepsilon.   \]
\end{enumerate}
\end{definition}

\begin{lemma}\label{lema: partition}
Given any $\delta>0$ and any $\varepsilon >0$, there exist an $\varepsilon$-regular covering of connected rectangles $\mathcal R_{\varepsilon}$ of $M$ with $\operatorname{diam}(R) < \delta$, for every $R \in \mathcal R_{\varepsilon}$.
\end{lemma}
\begin{proof}
As remarked by Chernov-Haskell \cite{CH}, in the non-uniformly hyperbolic case the expansion and contraction of the unstable and stable manifolds are not used to prove the existence of an $\varepsilon$-regular covering. The only properties used in the proof are the measurable dependence on $x \in M$, transversality between the unstable and stable foliations and the absolutely continuity property.

By our hypothesis, all these properties are satisfied for the pair of foliations $\mathcal F^{cs}$ and $\mathcal F^{u}$. Now the proof of this Lemma follows the same lines as the proof of Lemma $5.1$ from \cite{CH} just replacing stable by center-stable foliation.
\end{proof}

\begin{definition}
We say that a measurable set $A$ intersects a rectangle $\Pi$, \textit{leafwise} if
\[\mathcal F^u(w) \cap \Pi \subset A \cap \Pi, \text{ for any } w \in A\cap \Pi.\]
\end{definition}

\begin{lemma}\label{lemma:sub}
If $E$ is a set intersecting a rectangle $\Pi$ leafwise then the intersection $E \cap \Pi$ is a rectangle.
\end{lemma}
\begin{proof}
Let $x,y \in E \cap \Pi$ and set $z = W^u(x) \cap W^{cs}(y)$. Since $x,y \in \Pi$ and $\Pi$ is a rectangle we have $z \in \Pi$. Now, because the intersection is leafwise we have
\[z \in \mathcal F^u(x) \cap \Pi \subset \mathcal F^u(x) \cap E.\]
\end{proof}

\begin{lemma} \label{lemma: u-complete}
The set $\mathbb T^3 \setminus \mathcal C$ is u-saturated.
In particular, given any rectangle $\Pi$, if $X = \mathbb T^3 \setminus \mathcal C$ then $X$ intersects $\Pi$ leafwise.
\end{lemma}
\begin{proof}
Let $x \in \mathbb T^3 \setminus \mathcal C$ and assume that we can find $ y\in \mathcal F^u(x) \cap \mathcal C$. Then, there exists a closed segment $\gamma \subset \mathcal F^c(y)$ with $y\in \gamma$ and $\gamma \subset \mathcal C$. Take any $w \in \gamma \setminus \{y\}$ and consider $z = W^u(w) \cap W^c(x)$. 

Since we assumed that the center Lyapunov exponent of the linearization $A$ of $f$ is negative, the semi-conjugacy $h$ sends strong unstable leaves of $f$ to unstable leaves of $A$. Note that stable leaves of $f$ are not necessarily mapped to stable leaves of $A$.

If $h(y) = h(w)$, then 
 \[ \mathcal F^u(h(x))=\mathcal F^u(h(y)) = \mathcal F^u(h(w)) = \mathcal F^u(h(z)).\]
Thus we conclude that $h(z) \in \mathcal F^u(h(x)) \cap \mathcal F^c(h(x))$. Since $W^u(h(x)) \cap W^c(h(x)) = \{h(x)\}$ we can take $z$ close enough to $y$ and then we have $h(x) = h(z)$, contradicting the hypothesis that $x \notin \mathcal C$.
\end{proof}

\begin{lemma} \label{lemma: leafwise}
Given a rectangle $\Pi$ and $\beta > 0$, one can find $N_1>0$ such that for any $N' \geq N \geq N_1$ and $\beta$-almost every element $A \in \bigvee_{N}^{N'}f^k \alpha$, there exists a subset $E \subset A$, intersecting $\Pi$ leafwise, for which:
\[ \frac{m(E)}{m(A)} \geq 1-\beta .\]
\end{lemma}
\begin{proof}
The proof of this fact in our setting is exactly the same proof as for the nonuniformly hyperbolic case. Therefore we refer the reader to \cite{YP3} for a detailed demonstration. We point out that the proof does not need contraction of $\mathcal F^{cs}$, but it depends only on the expansion of $\mathcal F^u$.
\end{proof}

\subsection{Construction of the function $\theta$}

We now proceed to the most important part of the proof: the construction of the function $\theta$ satisfying the hypothesis of Lemma \ref{lema: bernoulli}. Such function $\theta$ should have the property that for most of the points $x$, the orbits of $x$ and of $\theta(x)$ have almost the same information (asymptotically) with respect to the partition in question. In particular, if we get good control on the distance $d(f^n(x), f^n(\theta(x)))$ we can expect that for most points $x$, $f^n(x)$ and $f^n(\theta(x))$ belong to the same partition element; this will be much clearer below. In order to get this control, we need to restrict to a compact set (where $h$ is uniformly continuous) and then use the fact that the center manifold of $A$ is uniformly contracting and $h$ sends center manifolds to center manifolds. This is what we do in the next lemma.

\begin{lemma} \label{lemma: distancia}
Given any $\varepsilon >0$ and any compact set $K$ in the essential injectivity domain of $h$ (see definition \ref{defi:X}) there exists 
$n_0 \in \mathbb N$ such that
for any two points $a \in K$, $b \in \mathcal F^{cs}(a) \cap K$, with $d(a,b) < \frac{1}{2}$ we have 
\[ d(f^n(a),f^n(b)) < \varepsilon, \text{ whenever } f^n(a),f^n(b) \in K \text{ and } n\geq n_0.\]
\end{lemma}
\begin{proof}
We split the proof in two cases.\\

\noindent {First case: $X=\mathbb T^3 \setminus \mathcal C$.}\\

Consider the lifts to the universal cover $\tilde{A}, \tilde{f} : \mathbb R^3 \rightarrow \mathbb R^3$ and $\tilde{h}:\mathbb R^3 \rightarrow \mathbb R^3$ the lift of the conjugacy, such that $\tilde{h}(0)=0$. We know that
\[\tilde{A}^n \circ \tilde{h} = \tilde{h} \circ \tilde{f}^n\]
for all $n$. Thus
\[(e^{\lambda^c_A})^nd(\tilde{h}(a),\tilde{h}(b)) \geq d(\tilde{A}^n \circ \tilde{h} (a), \tilde{A}^n \circ \tilde{h} (b)) = d(\tilde{h} \circ \tilde{f}^n(a),\tilde{h} \circ \tilde{f}^n(b)).\]
Since $h$ is a bounded distance from the identity, then if $d(a,b) < \delta$ we have:
\[d(\tilde{A}^n \circ \tilde{h} (a), \tilde{A}^n \circ \tilde{h} (b)) \leq   (e^{\lambda^c_A})^n \cdot D, \]
for a certain constant $D>0$.
By the uniform continuity of $h^{-1}$ inside $h(K)$ we can take $n_0$ big enough so that $n \geq n_0$ implies:
\[d(f^n(a), f^n(b)) < \varepsilon.\]

\noindent {Second case: $X=  \text{ set of atoms}$.}\\

In this case, we know by Theorem \ref{theo:main.B} that each leaf has only one atom, that is, for almost every $x \in \mathbb T^3$ we have $X \cap \mathcal F^c(x) = \{a_x \}$. Thus, given any two points $a,b \in X$, $a$ and $b$ are not collapsed by $h$ (since they do not belong to the same central leaf). Thus the proof of the first case works for this case as well.
\end{proof}

\begin{lemma} \label{lemma:key}
For any $\delta>0$, there exists $0<\delta_1 < \delta$ with the following property. Let $\Pi$ be a $\delta_1$-rectangle and $E$ a set intersecting $\Pi$ leafwise. Then we can construct a bijective function $\theta: E\cap \Pi \rightarrow \Pi$ such that for every measurable set $F \subset E \cap \Pi$ we have

\[ \frac{m^P_{\Pi}(\theta(F))}{m^P_{\Pi}(\Pi)} = \frac{m^P_{\Pi}(F)}{m^P_{\Pi}(E\cap \Pi)} , \]
and for every $x \in E\cap \Pi$ \[\theta(x) \in \mathcal F^{cs}(x).\]
\end{lemma}
\begin{proof}

Since $E$ intersects $\Pi$ leafwise by Lemma \ref{lemma:sub} we know that $E \cap \Pi$ is a sub-rectangle, and since the center stable foliation is absolutely continuous the intersection $\mathcal F^{cs}(x) \cap E \cap \Pi$ has positive Lebesgue measure.
Because $\mathcal F^{cs}(x) \cap E \cap \Pi$ and $\mathcal F^{cs}(x) \cap \Pi$ are both probability Lebesgue spaces (with the normalized measures) we can construct a bijection $\theta_0: \mathcal F^{cs}(x) \cap E \cap \Pi \rightarrow \mathcal F^{cs}(x) \cap \Pi$ preserving the normalized measures, that is, for any measurable subset $J \subset \mathcal F^{cs}(x) \cap E \cap \Pi$ we have

\begin{equation}\label{eq:cs}
\frac{m_f^{cs}(\theta_0(J))}{m_{\Pi}^P(\Pi)} = \frac{m_f^{cs}(J)}{m_{\Pi}^P(E\cap \Pi)} . \end{equation}

Now, given any $y \in E\cap \Pi$ we define (using that the intersection is leafwise and inside the rectangle) $\theta(y) \in \Pi$ by
\[\theta(y) := (\pi^u_{y,x})^{-1} \circ \theta_0 \circ \pi^u_{y,x}(y).\]

This $\theta: E \cap \Pi \rightarrow \Pi$ is well defined and $\theta(y) \in \mathcal F^{cs}(y) \cap \Pi$. 
By the definition of $m^P_{\Pi}$, given a measurable set $F \subset E\cap \Pi$ we have
\[m^P_{\Pi}(F) = \int m^{cs}_f(\pi^u_{y,x} (\mathcal F^{cs}(y) \cap F   ) ) dm^u(y) .\]
Thus,
\begin{align*}
m^P_{\Pi}(\theta(F)) & = \int m^{cs}_f(\pi^u_{y,x} (\mathcal F^{cs}(y) \cap \theta(F))) dm^u(y) \\
& =  \int m^{cs}_f(\pi^u_{y,x} \circ \theta(\mathcal F^{cs}(y) \cap F)) dm^u(y) \\
& =  \int m^{cs}_f(\theta_0 \circ \pi^u_{y,x}(\mathcal F^{cs}(y) \cap F)) dm^u(y).
 \end{align*}
Substituting \eqref{eq:cs} we have
\begin{align*}
m^P_{\Pi}(\theta(F)) & = \int m^{cs}_f(\theta_0 \circ \pi^u_{y,x}(\mathcal F^{cs}(y) \cap F)) dm^u(y)\\
& = \frac{m_{\Pi}^P(\Pi)}{m_{\Pi}^P(E\cap \Pi)} \int m^{cs}_f( \pi^u_{y,x}(\mathcal F^{cs}(y) \cap F)) dm^u(y) \\
& =  \frac{m_{\Pi}^P(\Pi)}{m_{\Pi}^P(E\cap \Pi)} \cdot m^P_{\Pi}(F)
 \end{align*}
 as we wanted to show.
\end{proof}

The following Lemma together with Theorems \ref{O70} and \ref{O702} conclude the proof of Theorem \ref{theo:main}.

\begin{lemma} \label{lema:final}
Let $\alpha$ be a finite partition with the property that each atom of $\alpha$ has piecewise smooth boundary. Then $\alpha$ is VWB.
\end{lemma}

The proof of Lemma \ref{lema:final} closely follows from the arguments already used by Pesin \cite{YP3} and Chernov-Haskell \cite{CH} with the technical difference that, by Lemma \ref{lemma:key}, the function $\theta$, constructed from a tubular intersection to a rectangle containing this intersection, does not preserve stable manifolds as in \cite{OW}, \cite{YP3} and \cite{CH}. Instead, $\theta$ preserves center-stable manifolds. Points belonging to the same center-stable manifold do not have the property of getting exponentially close to each other as we iterate the dynamics, thus we cannot directly say that given $\varepsilon >0$ a large set of pairs of points on the same center-stable manifold will asymptotically visit the same atoms.\\

To overcome this difficulty we use Theorem \ref{theo:main.B}. Theorem \ref{theo:main.B} says that either $m(\mathcal C)=0$ and then we have defined $X = \mathbb T^3 \setminus \mathcal C$, or $m(\mathcal C)=1$ and then we can take a full measure set $X \subset \mathbb T^3$ intersecting almost every center-leaf in exactly one point.
By Lemma \ref{lemma: distancia}, we can choose arbitrarily large compact sets $K$ such that any pair of point $x,y \in K \subset X$ have the property that if $y \in \mathcal F^{cs}(x)$ then the distance between $f^n(x)$ and  $f^n(y)$ is very small every time both of them visits the set $K$ simultaneously. The point is that since we can take $K$ arbitrarily large and $f$ is ergodic, the set of natural numbers $\{n : f^n(x), f^n(y) \in K\}$ has arbitrarily large density, this will allow us to conclude that indeed for a large set of points $x \in X$ the Cesaro-mean of Lemma \ref{lema: bernoulli} are indeed arbitrarily small.

\subsection{Proof of Lemma $4.12$}

As explained on the last paragraph of last section, the technical difference of our case is that we have to restrict ourselves to large compact sets where points on the same center-stable leaves behave well. Therefore, we will keep notations similar to the notation used by Chernov-Haskell so that we can omit some calculations, which are equal to the (nonuniformly) hyperbolic case, and refer the reader to \cite{CH} for the detailed estimates.\\

 Let $\alpha = \{A_1,...,A_b\}$ be a finite partition of $M$ where each element has piecewise-smooth boundaries. Thus, we can take a constant $D_0$ such that for any $\varepsilon>0$ and any $i=1,...,b$, the $\varepsilon$-neighborhood of $A_i$, denoted by $O_{\varepsilon}(A_i)$, has measure less than $D_0\cdot \varepsilon$.\\
 
 Let $\varepsilon >0$ and, as in \cite{CH}, we take $\delta = \varepsilon^4$. Let $ \{R_1,...,R_k\}$ be a $\delta$-regular covering of $M$ and define the partition $\pi=\{R_0,R_1,...,R_k\}$ of $M$ by taking
 \[R_0:= M \setminus \bigcup_{i=1}^{k}R_i.\]
 By definition of $\delta$-regular covering we have $m(R_0) < \delta$ and, for each $1\leq i \leq k$, we can take a set $G_i \subset R_i$ which satisfies condition $(2)$ of Definition \ref{defi:regular}.  Since $f$ has the $K$-property there exists a natural number $N$ such that for all $N_1>N_0\geq N$ and for $\delta$-almost every atom $A \in \bigvee_{N_0}^{N_1}f^i\alpha$ has the property that for all $R \in \pi$
 \[\left|\frac{m(R\cap A)}{m(R)m(A)}-1 \right|< \delta,\]
 that is,
  \begin{equation}\label{17}
  \left|\frac{m(R/ A)}{m(R)}-1 \right|< \delta,\end{equation}
 where $m(\cdot / A)$ denotes the measure $m$ conditioned on $A$.

Now, fix natural numbers $N_1>N_0>N$ and $n>0$. We want to prove that for a certain constant $D>0$, $D\cdot \varepsilon$-almost every atom of $\bigvee_{N_0}^{N_1}f^i\alpha$ satisfies
\[\bar{d}(\{f^{-i}\alpha \}_{1}^n  , \{f^{-i}\alpha | A \}_1^n  ) \leq D\cdot \varepsilon.\]

As in \cite{CH}, the first step is to identify the set of ``bad'' elements which will have measure at most $D \cdot \varepsilon$.\\

\noindent{\bf First bad set $B_1$: } Let $B_1$ be the union of all atoms of $\bigvee_{N_0}^{N_1}f^i\alpha$ which do not satisfy \eqref{17}.\\

\noindent {\bf Second bad set $B_2$: } Denote by $F_2 = \bigcup_{i=1}^{k}R_i \setminus G_i$. Take $B_2$ to be the union of atoms $A \in \bigvee_{N_0}^{N_1}f^i\alpha$ such that either
\[m(F_2 / A) > \delta^{1/2}\]
or
\[\sum_{i=1}^{k}\frac{m_{R_i}^P(A\cap F_2)}{m(A)} > \delta^{1/2}. \]
It is easy to show that $m(B_2) < c_1 \cdot \delta^{1/2}$ for a certain constant $c_1>0$ (\cite{CH}, Pg.$22$).\\

\noindent {\bf Third bad set $B_3$: } By Lemma \ref{lemma: leafwise}, given a rectangle $\Pi$ and any $\beta >0$ we can find $\tilde{N}_1>0$ such that for all $\tilde{N}' \geq \tilde{N} \geq \tilde{N}_1$ and $\beta$-almost every element $A \in \bigvee_{\tilde{N}}^{\tilde{N}'}f^i \alpha$ , there exists a subset $E\subset A$, intersecting the rectangle $\Pi$ leafwise and for which
\[\frac{m(E)}{m(A)} \geq 1-\beta.\]

Thus, by taking $\beta$ small enough and $N$ big enough, the set $F_3$ of all points $x \in M \setminus R_0$ which lies in a non-leafwise intersection (with respect to the rectangle of $\pi$ containing $x$) satisfies
\[m(F_3)<\delta.\]
Let $B_3$ denotes the union of all atoms $A$ of $\bigvee_{N_0}^{N_1}f^i\alpha$ for which
\[m(F_3 / A) > \delta^{1/2}.\]
Then it follows that $m(B_3)<c_1\cdot \delta^{1/2}$ (\cite{CH}, Pg.$22$).\\

By the estimates on the measures of the bad sets we have that the union of all atoms on the complement of $B_1\cup B_2 \cup B_3$ is at least $1-c_1\varepsilon$ for a certain constant $c_1>0$.\\

Let $A$ be an atom of $\bigvee_{N_0}^{N_1}f^i\alpha$ which is in the complement of $B_1$, $B_2$ and $B_3$. By Lemma \ref{lemma:key}, for each $1\leq i \leq k$ with $A\cap R_i \ne \emptyset$, we can construct a bijective function $\theta_i: A \cap R_i \cap F^c_3 \rightarrow R_i$ satisfying
\[ \frac{m^P_{R_i}(\theta(B))}{m^P_{R_i}(R_i)} = \frac{m^P_{R_i}(B)}{m^P_{R_i}(A \cap R_i \cap F^c_3)} , \]
for every measurable set $B \subset A \cap R_i \cap F^c_3$ and
\[\theta(x) \in \mathcal F^{cs}(x),\]
for every $x\in A \cap R_i \cap F^c_3$.
If we denote $\mu^P_{R_i}:=m^P_{R_i}(\cdot / R_i)$ we can rewrite the previous equality as,
\[\mu^P_{R_i}(B / A\cap F^c_3) = \mu^P_{R_i}(\theta_i(B)),\]
for every measurable set $B \subset A \cap R_i \cap F^c_3$.

\begin{lemma}\label{lemma:E}
There exists a constant $c_2>0$ and a measurable set $E_2 \subset M$ with
\[m(E_2) < c_2\cdot \varepsilon,\]
and such that for each $1\leq i \leq k$ with $A\cap R_i \ne \emptyset$ the function $\theta_i$ satisfies
\[\left| \frac{m(B/A\cap R_i)}{m(\theta_i(B)/ R_i)} -1 \right|< c_2\cdot \varepsilon,\]
for any $B \subset A\cap R_i \cap E_2^c$. 
\end{lemma}

For a detailed proof of this Lemma see \cite{CH} Pg.$24$ and $25$.\\

Now define the function $\theta:A \rightarrow M$ as $\theta(x) = \theta_i(x)$ if $x \in A\cap R_i \cap F_3^c$ for some $1\leq i \leq k$, and $\theta(x) = x$ otherwise.

\begin{lemma}\label{lemma:ultimo}
$\theta:A\rightarrow M$ is $c_2\cdot \varepsilon$-measure preserving.
\end{lemma}
\begin{proof}
See \cite{CH}, Pg.$26$.
\end{proof}

To conclude the proof that $\alpha$ is VWB there is one final step. Recall that the function $\theta:A \rightarrow M$ constructed above has the property of being $c_2\cdot \varepsilon$-measure preserving and 
\[\theta(x) \in \mathcal F^{cs}(x) \cap R_i\]
for any $x\in A\cap R_i $, $1\leq i \leq k$. To use Lemma \ref{lema: bernoulli} and finish the proof we still need to prove that the Cesaro sum that appears in Lemma \ref{lema: bernoulli} is small for a large set of points $x$. Here is where we need to restrict ourselves to a large compact set where we can apply Lemma \ref{lemma: distancia} for the recurrent pairs of points in this set.\\

Take an arbitrary $\zeta<1$ and consider $K \subset E_2^c \cap X$ (see definition \ref{defi:X}) a compact set with
\[m(K) > \zeta \cdot m(E_2^c).\]

Take $\kappa$ the set of points of $K$ such that the past and future Birkhoff averages coincide and converge to $m(K)$, that is: $x \in \kappa$ if $x \in K$ and

\[\lim_{n\rightarrow -\infty}\frac{1}{|n|}\sum_{j=0}^{n-1}\chi_{K} (f^j(x))  = \lim_{n \rightarrow \infty} \frac{1}{n}\sum_{j=0}^{n-1}\chi_{K} (f^j(x)) = m(K). \]

By Birkhoff's Theorem we know that $m(\kappa) = m(K)$. Take $P:= \kappa \cap \theta^{-1}(\kappa)$. Since $\theta$ is $c_2\varepsilon$-measure preserving, taking $\zeta$ close enough to one we have that
\[m(E_2^c \setminus P) \leq 2c_2 \cdot \varepsilon.\]

Now, observe that by Lemma \ref{lemma: distancia} we can take $n_0\geq 0$ such that for any $ x\in \kappa$ we have
\[d(f^n(x),f^n(\theta(x))) < \varepsilon \]
for all $n \geq n_0$ with $f^n(x),f^n(\theta(x)) \in \kappa$. \\

Let $l_i(x)$ be the name of $x$ with respect to the sequence of partitions $\xi_i:=f^{-i}\alpha|A$ and $m_i(x)$ the name of $x$ with respect to the partitions $\eta_i = f^{-i}\alpha$.
If $x\in \kappa$, $i\geq n_0$ and $e(l_i(x) - m_i(\theta(x))) = 1$ then either:\\
\begin{itemize}
\item $f^i(x) \notin \kappa$ or $f^i(\theta(x))\notin \kappa$; or
\item $d(f^i(x),f^i(\theta(x))) < \varepsilon$ and then,
\[ d(f^i(x), \partial A_{l_i(x)}) < \varepsilon \Rightarrow f^i(x) \in O_{\varepsilon}(A_{l_i(x)}).\]
\end{itemize}
\quad\\
Take \[O_{\varepsilon} := \bigcup_{i=1}^{k}O_{\varepsilon}(A_i)\] and consider 
\[J^x= \{ j  \in \mathbb N \text{ such that } f^j(x) \notin \kappa \text{ or } f^j(\theta(x)) \notin \kappa \},\] 
\[J^x_n := J ^x\cap [1,n].\]
By the definition of the function $e$ we have 
\begin{eqnarray*}
\frac{1}{n} \sum_{i=1}^{n} e(l_i(x)  -  m_i(\theta(x))) \leq \frac{1}{n}   \sum_{j=1}^n \chi_{O_{\varepsilon}} (f^j(x))  +   \frac{1}{n} \#J^x_n,
\end{eqnarray*}
By ergodicity the right side converges to $[m(O_{\varepsilon}) + \operatorname{dens}(J^x) ]$ for almost every $x$. Since we can take $K$ with arbitrarily large measure, and since $m(O_{\varepsilon}) < D_0\cdot \varepsilon$ it follows that, there exist a set $\hat{P} \subset P$ with measure $m(\hat{P})>1-c_3\varepsilon$ such that for all $x \in \hat{P}$ 
\[\frac{1}{n} \sum_{i=1}^{n} e(l_i(x)  -  m_i(\theta(x))) \leq c_3 \cdot \varepsilon,\]
for a certain constant $c_3>0$ and $n$ large.
Applying Lemma \ref{lema: bernoulli} we conclude that
\[\bar{d}(\{\xi_i\}_1^n,\{\eta_i\}_1^n ) \leq c_4 \varepsilon,\]
for a certain constant $c_4$ which do not depend on $\varepsilon$.
Since $\varepsilon>0$ is arbitrary it follows that $\alpha$ is $\operatorname{VWB}$ as we wanted to show. \hfill $\square$

\bibliographystyle{plain}
\bibliography{Referencias.bib}
\end{document}